\documentclass{article}

\usepackage{amsthm, amsmath, amsfonts}

\usepackage{color}

\newtheorem{thm}{Theorem}
\newtheorem{define}{Definition}[section]
\newtheorem{lemma}[define]{Lemma}
\newtheorem{prop}[define]{Proposition}

\newtheorem{thmcoro}[thm]{Corollary}

\usepackage{tikz}
\usepackage{verbatim}
\usepackage{caption}

\usepackage{thmtools, thm-restate}
\usepackage{hyperref}
\hypersetup{
  linkcolor=green,
  citecolor=green,
  colorlinks=false,
}

\title{Hamilton cycles, minimum degree and bipartite holes}

\begin{document}

\author{
  Colin McDiarmid\\
  \texttt{cmcd@stats.ox.ac.uk}
  \and
  Nikola Yolov\\
  \texttt{nikola.yolov@cs.ox.ac.uk}
}
\maketitle

\begin{abstract}
We present a tight extremal threshold for the existence of Hamilton cycles
in graphs with large minimum degree and without a large ``bipartite hole``
(two disjoint sets of vertices with no edges between them).
This result extends Dirac's classical theorem,
and is related to a theorem of Chv\'atal and Erd\H{o}s.

In detail, an \emph{$(s, t)$-bipartite-hole} in a graph $G$
consists of two disjoint sets of vertices $S$ and $T$ with $|S|= s$ and $|T|=t$
such that there are no edges between $S$ and $T$;
and $\widetilde{\alpha}(G)$ is the maximum integer $r$
such that $G$ contains an $(s, t)$-bipartite-hole
for every pair of non-negative integers $s$ and $t$ with $s + t = r$.
Our central theorem is that a graph $G$ with at least $3$ vertices is Hamiltonian
if its minimum degree is at least $\widetilde{\alpha}(G)$.

From the proof we obtain a polynomial time algorithm that either finds a Hamilton cycle
or a large bipartite hole.
The theorem also yields a condition for the existence of $k$ edge-disjoint Hamilton cycles.
We see that for dense random graphs $G(n,p)$,
the probability of failing to contain many edge-disjoint Hamilton cycles
is $(1 - p)^{(1 + o(1))n}$.
Finally, we discuss the complexity of calculating and approximating
$\widetilde{\alpha}(G)$.
\end{abstract}

%%%%%%%%%%%%%%%%%%%%%%%%%%%%%%%%%%%%%%%

\section{Introduction and statement of results}

Hamilton cycles are one of the central topics in graph theory,
see for example~\cite{bondy_murty}.
The problem of recognising the existence of a Hamilton cycle in a graph
is included in Karp's 21 NP-complete problems \cite{karp_21_problems}.
Recall that $\delta(G)$ denotes the minimum degree $d(v)$ of a vertex $v$ in $G$.
An early result by Dirac \cite{dirac} states:

\begin{thm}[Dirac's Theorem]
  \label{thm:Dirac}
  A graph $G$ with $n \ge 3$ vertices is Hamiltonian
  if $\delta(G) \geq n/2$.
\end{thm}

The theorem is sharp,
since the disjoint union of two complete $n$-vertex graphs
has minimum degree $n-1$ and it does not contain a Hamilton cycle.
This example contains a large bipartite hole,
that is two disjoint sets of vertices with no edge between them.
It is natural to ask if such a hole is necessary
to construct a non-Hamiltonian graph with large minimum degree.
We show that indeed this is the case.

Given disjoint sets $S$ and $T$ of vertices in a graph,
we let $E(S,T)$ denote the set of edges with one end in $S$ and one in $T$.

\begin{define}
  An \emph{$(s, t)$-bipartite-hole in a graph $G$}
  consists of two disjoint sets of vertices $S$ and $T$ with
  $|S|= s$ and $|T|=t$ such that  $E(S, T) = \emptyset$.
  We define the \emph{bipartite-hole-number}  $\widetilde{\alpha}(G)$ to be
  the least integer $r$ which may be written as $r=s+t-1$ for some positive
  integers $s$ and $t$ such that $G$ does not contain an $(s, t)$-bipartite-hole.
\end{define}
\noindent

An equivalent definition of $\widetilde{\alpha}(G)$ is the maximum integer $r$
such that $G$ contains an $(s, t)$-bipartite-hole
for every pair of non-negative integers $s$ and $t$ with $s + t = r$.
Observe that $\widetilde{\alpha}(G) = 1$ if and only if $G$ is complete,
and $\widetilde{\alpha}(G) \ge \alpha(G)$,
where $\alpha(G)$ is the stability number of $G$.
Also note that for $1 \leq a \leq b$, we have
$\widetilde{\alpha}(K_{a,b}) = b$ and
$\widetilde{\alpha}(\overline{K_{a,b}}) = \min\{b+1, 2a+1\}$.
(Here $K_{a,b}$ denotes the complete bipartite graph with parts of sizes $a$ and $b$,
and $\overline{G}$ denotes the complement of $G$.)

The following is our main theorem.
It arose from our investigations of the random perfect graph $P_n$,
where we wished to show that $P_n$ is Hamiltonian with failure probability
$e^{-\Omega(n)}$, see~\cite{rpg}.

\begin{restatable}{thm}{extremallemma}
  \label{thm:extremal}
  A graph $G$ with at least 3 vertices is Hamiltonian if
  $\delta(G) \ge \widetilde{\alpha}(G)$.
\end{restatable}

This result is sharp in the sense that for every positive integer $r$
there is a non-Hamiltonian graph with $\delta(G) = r = \widetilde{\alpha}(G) -1$.
An example is $G = K_{r, r+1}$, where $\delta(G)=r$ and $\widetilde{\alpha}(G)=r+1$.
Theorem \ref{thm:extremal} generalises Theorem~\ref{thm:Dirac} of Dirac.
Indeed, a graph $G$ with $\delta(G) \ge n/2$ has no
$(1, \lfloor n/2\rfloor)$-bipartite-hole,
and hence $\delta(G) \ge n/2 \ge \widetilde{\alpha}(G)$.
Also,
Theorem \ref{thm:extremal} can be extended to provide a sufficient
condition for the existence of many edge-disjoint Hamilton cycles;
and in fact the next result will be deduced quickly from Theorem~\ref{thm:extremal}.

\begin{restatable}{thm}{disjointcycles}
  \label{thm:disjointcycles}
  Let $r \ge 0$ be an integer,
  and let $G$ be a graph with at least 3 vertices such that
  $\delta(G) \ge (r + 1)\widetilde{\alpha}(G) + 3r$.
  Then $G$ contains $r+1$ edge-disjoint Hamilton cycles.
\end{restatable}

\noindent
Note that by setting $r=0$ in Theorem \ref{thm:disjointcycles}
we regain Theorem \ref{thm:extremal}.

It is perhaps not surprising that determining $\widetilde{\alpha}(G)$
is NP-hard and that it is hard to approximate, see Section~\ref{sec.hard} below.
However, Theorem \ref{thm:extremal} can be made algorithmic.

\begin{restatable}{thm}{algorithm}
  \label{thm:algorithm}
  There is an algorithm which, on input a graph $G$ with $n \ge 3$ vertices,
  in $O(n^3)$ time outputs either a Hamilton cycle or
  a certificate that $\widetilde{\alpha}(G) > \delta(G)$.
\end{restatable}

Theorem \ref{thm:disjointcycles} can also be made algorithmic.
One can repeatedly use the algorithm in Theorem \ref{thm:algorithm}
to find a Hamilton cycle, remove its edges from $G$ and repeat,
or if no cycle is found, output a certificate that $\widetilde{\alpha}(G)$ is large.
This yields:

\begin{restatable}{thm}{cor}
\label{thm:cor}
  There is an algorithm that, on input a graph $G$ with $n \geq 3$ vertices,
  in $O(n^4)$ time outputs a non-negative integer $r$,
  a collection of $r$ edge-disjoint Hamilton cycles of $G$,
  and a certificate that $\widetilde{\alpha}(G) > \frac{\delta(G) - 3r}{r+1}$.
\end{restatable}

Containing a large bipartite hole is not a certificate for the absence of
Hamilton cycles;
there are Hamiltonian graphs for which the algorithm will stop before outputting
a Hamilton cycle, which is to be expected,
since deciding whether or not a graph is Hamiltonian is NP-complete.

We conclude the paper by applying Theorem~\ref{thm:disjointcycles} to
show quickly that for a sufficiently dense random graph $G$,
the probability of $G$ failing to contain many edge-disjoint Hamilton cycles
is well-estimated by the probability that $G$ contains a vertex
with too small degree ($<2r$), or indeed contains an isolated vertex.

\begin{thm}
  \label{thm:random_disjoint}
  Let $0<\epsilon<1$, let $0 \leq p=p(n) \leq 1-\epsilon$,
  and let $r=r(n)$ be a positive integer.
  If $\frac{p(n)\sqrt{n}}{r(n)\log n} \to \infty$ as $n \to \infty$,
  then the probability that $G(n, p)$ fails to contain at least $r$ edge-disjoint
  Hamilton cycles is $(1-p)^{(1+o(1))n}$.
\end{thm}

\noindent
Setting $r = 1$ we obtain:
\begin{thmcoro}
  \label{thm:random_cycle}
  If $p(n)\sqrt{n}/ \log n \to \infty$
  as $n \to \infty$,
  then the probability that $G(n, p)$ fails to be Hamiltonian is $(1-p)^{(1 + o(1))n}$.
\end{thmcoro}

%%%%%%%%%%%%%%%%%%%%%%%%%%%%%%%%%%

\section{Related work}

Finding sufficient conditions for the existence of Hamilton cycles
has been an active area of research for more than sixty years.
Among the most well-known conditions are Dirac's Theorem \cite{dirac},
Theorem~\ref{thm:Dirac};
and a generalisation by Ore \cite{ore}, which states that
an $n$-vertex graph $G$ is Hamiltonian if
$d(u) + d(v) \ge n$ for any pair of non-adjacent vertices $u$ and $v$.
These were further generalised by Bondy and Chv\'atal, and others,
see the book by Bondy and Murty~\cite{bondy_murty} and see~\cite{gould,li} for surveys. 
Both conditions are further generalised by Fan \cite{fan},
where he proved that a $2$-connected graph $G$ of order $n$
is Hamiltonian if $\max(d(u), d(v)) \ge n/2$
for every pair of nonadjacent vertices $u, v$ with distance $2$.
See \cite{survey} for a survey.

One of these generalisations, by Chv\'atal and Erd\H{o}s \cite{chvatal_erdos},
has a sharp condition close to the one in this paper.
We denote the vertex connectivity of $G$ by $\kappa(G)$
and the number of vertices of $G$ by $v(G)$.

\begin{thm}[Chv\'atal-Erd\H{o}s Theorem]
  \label{thm:erdos}
  A graph $G$ with at least 3 vertices is Hamiltonian if $\kappa(G) \ge \alpha(G)$.
\end{thm}
There are interesting connections between Theorems \ref{thm:extremal}
and \ref{thm:erdos}, and between the parameters
$\kappa$, $\delta$, $\alpha$ and $\widetilde{\alpha}$.
For example, $\kappa(G) \le \delta(G) \le v(G) - \alpha(G)$
and $\alpha(G) \le \widetilde{\alpha}(G) \le v(G) - \kappa(G)$.
Furthermore, we will see in Lemma~\ref{thm:extremal_preliminary}
that $\kappa(G) \ge \delta(G) - \widetilde{\alpha}(G) + 2$.

Comparing Theorems~\ref{thm:extremal}
and~\ref{thm:erdos},  neither condition implies the other.
Here is an example of a graph $G$ that meets the conditions of Theorem
\ref{thm:extremal} but not Theorem \ref{thm:erdos}.
It has vertex set $V(G) = \{a\} \cup B \cup C \cup D$, such that $|B| = k+\ell$,
$|C| = k$, $|D| = \ell+1$, and all these sets are disjoint.
All edges between $\{a\}$ and  $B$, between $B$ and $C$,
and between $C$ and $D$ are present, $B$ and $D$ are complete,
and $C$ is independent.
It is easy to see that $k = \kappa(G) < \alpha(G) = k + 1$,
and $\delta(G) = k + \ell \ge \max\{k+1, 2\ell+3\} = \widetilde{\alpha}(G)$
for $\ell \ge 1$ and $k \ge \ell+3$.
In the other direction, $C_5$ satisfies $\kappa=2=\alpha$
but $\delta=2<3=\widetilde{\alpha}$.

A more recent related result is by Hefetz, Krivelevich and Szab\'o
\cite{Krivelevich_hamilton}.
Roughly speaking, the authors prove that expanding graphs
without large bipartite holes are Hamiltonian.
Their results cover a wide range of graphs, including relatively sparse graphs.
Compared to \cite{Krivelevich_hamilton},
we focus on tight extremal thresholds, simple self-contained proofs
and the right conditions for edge-disjoint Hamilton cycles.

Hamilton cycles in random graphs have been well-studied,
see for example~\cite{Frieze_BCC,bollobas}.
In \cite{komlos_szemeredi} Koml\'os and Szemer\'edi prove that
if
\[
p = p(n) = \frac{k(n)}{{n \choose 2}}; \hspace{10pt}
k(n) = \frac12n \log n + \frac12 n \log \log n + c_nn,
\]
then
\[
\lim_{n \to \infty} \mathbb{P}(G(n, p) \text{ is Hamiltonian}) =
\begin{cases}
  0              &\mbox{if } c_n \to -\infty\\
  e^{-e^{-2c}}   &\mbox{if } c_n \to c \\
  1              &\mbox{if } c_n \to \infty.
\end{cases}
\]
Frieze proves in \cite{Frieze_ham_bipartite} a similar result
for random bipartite graphs.
The evolutionary process $G_{n,t}$ is defined as follows:
$G_{n,0}$ is the empty graph on $n$ vertices and
$G_{n,k+1}$ is obtained from $G_{n,k}$ by adding an edge uniformly at random.
Ajtai, Koml\'os and Szemer\'edi \cite{hitting_time_one}
and Bollob\'as \cite{hitting_time_two}
showed that with high probability the hitting time for Hamiltonicity
equals the hitting time for minimal degree at least two.

%%%%%%%%%%%%%%%%%%%%%%%%%%%%%%%

\section{Extremal condition for Hamilton cycle}
A necessary condition for a graph to be Hamiltonian is to be 2-connected,
so Theorem \ref{thm:extremal} implies that every graph $G$
with $\delta(G) \ge \widetilde{\alpha}(G)$ is 2-connected.
We give one preliminary lemma before proving
Theorems \ref{thm:extremal} and \ref{thm:algorithm}.

\begin{lemma}
  \label{thm:extremal_preliminary}
  The following holds for every graph $G$:
  \[
  \kappa(G) \ge \delta(G) +2  - \widetilde{\alpha}(G).
  \]
\end{lemma}

\begin{proof}
  Suppose for a contradiction the vertices $v_1$ and $v_2$ are separated by a set $S$ of size less than $\delta(G) + 2 - \widetilde{\alpha}(G)$.
  Let $s$ and $t$ be positive integers
such that $\widetilde{\alpha}(G) + 1 = s + t$
and $G$ has no $(s, t)$-bipartite-hole.
Then $\frac{\widetilde{\alpha}(G) + 1}{2} \le \max(s, t) \le \widetilde{\alpha}(G)$.
Now the closed neighbourhoods $N[v_i]$ satisfy
$|N[v_i] \setminus S| \ge \delta(G) + 1 - |S| \ge \widetilde{\alpha}(G) \ge \max(s, t)$.
  The sets $N[v_1] \setminus S$ and $N[v_2] \setminus S$ are disjoint
  because $S$ is a separator,
  but $|N[v_1] \setminus S| \ge s$ and $|N[v_2] \setminus S| \ge t$,
  so there is an edge between them and $S$ does not separate $v_1$ from $v_2$,
  a contradiction.
\end{proof}

As an aside before proving Theorem~\ref{thm:extremal},
suppose the graph $G$ with at least 3 vertices satisfies
$\delta(G) \ge 2 \widetilde{\alpha}(G) - 2$.
Then
\[
\kappa(G) \ge \delta(G) + 2 - \widetilde{\alpha}(G)
\ge \widetilde{\alpha}(G) \ge \alpha(G).
\]
Hence the conditions of the Chv\'atal-Erd\H{o}s Theorem are met, and so $G$ is Hamiltonian.

\extremallemma*
\begin{proof}
  If $\widetilde{\alpha}(G) = 1$, then $G$ is complete,
  and so $G$ is Hamiltonian.
  Thus we may suppose that $\widetilde{\alpha}(G) \ge 2$.
  We will show that if $P$ is a maximal length
  path in $G$, then $G[V(P)]$ is Hamiltonian.
  This, together with the connectedness of $G$
  following from  Lemma \ref{thm:extremal_preliminary},
  is enough to complete the proof.

  Indeed, suppose $P$ is a maximal length path in $G$, $n = v(P)$,
  and label the vertices in $V(P)$ with $[n] := \{1, \ldots, n\}$
  in the order they appear in the path, after choosing an arbitrary orientation.
  We may assume that vertices 1 and $n$ are not adjacent.
  For a set $S \subseteq V(P)$,
  define $S^+$ to be the set of successors $x^+$ of elements $x$ in $S$, and
  define $S^-$ to be the set of predecessors~$x^-$.
  We leave $S^+$ undefined if $n \in S$ and $S^-$ is undefined if $1 \in S$.

  We now describe three situations when $P$ can be closed to form a cycle.
  The first yields a standard proof of Dirac's and Ore's theorems,
  the second involves `non-crossing' edges from the end vertices,
  and the third involves `crossing edges'.
  \begin{description}
    \item{(a)}
      If for some $j \in (1,n)$ we have $j \in N(1)$ and $j^- \in N(n)$,
      then $1j-nj^--1$ is a spanning cycle of $V(P)$
      (where we follow the path $P$ from $j$ to $n$ and from $j^-$ to $1$).
      See Figure~1.%\ref{fig:single_flip} doesn't work for some strange reason
      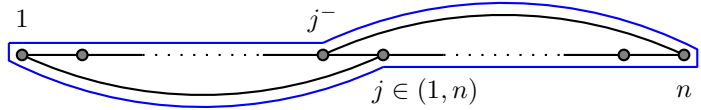
\begin{figure}[h]
        \label{fig:single_flip}
        \centering
        \begin{tikzpicture}[thick,scale=0.8]
    \draw                           (5,0)             arc  (115:65:7.1);
    \draw                           (0,0)             arc  (-115:-65:7.1);
    \draw \foreach \x in {0,1} {
      (\x+0,0)
      node [circle, draw, fill=black!50, inner sep=0pt, minimum width=4pt] {}
      -- (\x+1,0)
    };
    \draw [loosely dotted]          (2,0)             --   (4,0);
    \draw                           (4,0)             --   (5,0);
    \draw \foreach \x in {5,...,6}  {
      (\x+0,0)
      node [circle, draw, fill=black!50, inner sep=0pt, minimum width=4pt] {}
      --   (\x+1,0)};
    \draw [loosely dotted]          (7,0)             --   (9,0);
    \draw \foreach \x in {9,10}    {
      (\x+0,0)          --   (\x+1,0)
      node [circle, draw, fill=black!50, inner sep=0pt, minimum width=4pt] {}
    };

    \draw [blue]                     (0,0.2)           --   (5,0.2);
    \draw [blue]                     (5,0.2)           arc  (115:63:7.1);
    \draw [blue]                     (11.22,0.1)       --   (11.22,-0.21);
    \draw [blue]                     (6,-0.2)          --   (11.22,-0.2);
    \node  at (0,0.6) {$1$};
    \node  at (5,0.6) {$j^-$};
    \node  at (6.7,-0.6) {$j \in (1, n)$};
    \node  at (11,-0.6) {$n$};

    \draw [blue]                     (6,-0.2)          arc  (-65:-117:7.1);
    \draw [blue]                     (-0.22,-0.1)      --   (-0.22,0.2);
    \draw [blue]                     (-0.22,0.2)       --   (0,0.2);
\end{tikzpicture}
        \caption{Single flip}
      \end{figure}

    \item{(b)}
      If for some $k \in (1, n)$ there exist
      $i \in N(1) \cap (1, k]$ and $j \in N(n) \cap [k, n)$
      such that $i^-$ is adjacent to $j^+$,
      then $1 - i^-j^+ - nj-i1$ is a spanning cycle of $V(P)$.
      Here we may have $i=j$;
      see Figure~2.%Same here \ref{fig:nested_flip}.
      \begin{figure}[h]
        \label{fig:nested_flip}
        \centering
        \begin{tikzpicture}[thick,scale=0.8]
    \draw                           (0,0)             arc  (-115:-65:5.91);
    \draw                           (4,0)             arc  (115:65:7.1);
    \draw                           (9,0)             arc  (-115:-65:5.91);
    \draw {
      (0,0)
      node [circle, draw, fill=black!50, inner sep=0pt, minimum width=4pt] {}
      -- (1,0)
    };
    \draw [loosely dotted]          (1,0)             --   (3,0);
    \draw                           (3,0)             --   (4,0);
    \draw \foreach \x in {4,5}  {
      (\x+0,0)
      node [circle, draw, fill=black!50, inner sep=0pt, minimum width=4pt] {}
      --   (\x+1,0)};
    \draw [loosely dotted]          (6,0)             --   (8,0);
    \draw                           (8,0)             --   (9,0);
    \draw \foreach \x in {9,10}  {
      (\x+0,0)
      node [circle, draw, fill=black!50, inner sep=0pt, minimum width=4pt] {}
      --   (\x+1,0)};
    \draw [loosely dotted]          (11,0)            --   (13,0);
    \draw                           (13,0)            --   (14,0)
      node [circle, draw, fill=black!50, inner sep=0pt, minimum width=4pt] {}
      ;
    \draw [blue]                     (0,0.2)           --   (4,0.2);
    \draw [blue]                     (4,0.2)           arc  (115:65:7.1);
    \draw [blue]                     (10,0.2)          --   (14.22,0.2);
    \draw [blue]                     (14.22,0.21)      --   (14.22,-0.1);
    \draw [blue]                     (9,0.-0.2)        arc  (-115:-62.5:5.91);
    \draw [blue]                     (5,0.-0.2)        --   (9,0.-0.2);
    \node  at (0,0.6) {$1$};
    \node  at (4,0.6) {$i^-$};
    \node  at (5.7,-0.6) {$i \in (1, k]$};
    \node  at (8.4,-0.6) {$j \in [k, n)$};
    \node  at (10,0.6) {$j^+$};
    \node  at (14,0.6) {$n$};

    \draw [blue]                     (5,-0.2)          arc  (-65:-117.5:5.91);
    \draw [blue]                     (-0.22,-0.1)      --   (-0.22,0.2);
    \draw [blue]                     (-0.22,0.2)       --   (0,0.2);

\end{tikzpicture}
        \caption{Double nested flip}
      \end{figure}
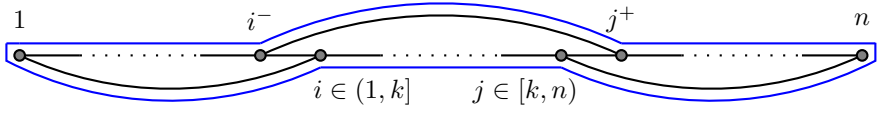

    \item{(c)}
      If for some $k \in (1, n)$ there exist
      $i \in N(1) \cap [k, n)$ and $j \in N(n) \cap [1, k)$
      such that $i^+$ is adjacent to $j^+$,
      then $1-jn-i^+j^+-i1$ is a spanning cycle of $V(P)$.
      Here we may have $j^+ =i$; see Figure~3.%\ref{fig:single_flip}.
      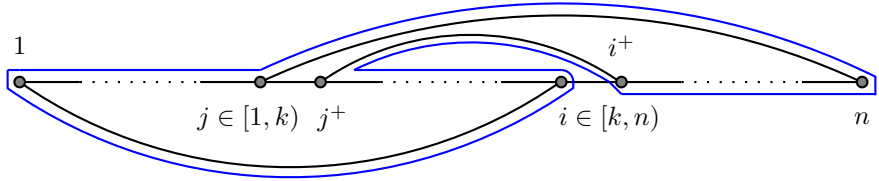
\begin{figure}[h]
        \label{fig:cross_flip}
        \centering
        \begin{tikzpicture}[thick,scale=0.8]
    \draw                           (4,0)             arc  (115:65:11.82);
    \draw                           (5,0)             arc  (125:55:4.34);
    \draw                           (0,0)             arc  (-125:-55:7.83);
    \draw {
      (0,0)
      node [circle, draw, fill=black!50, inner sep=0pt, minimum width=4pt] {}
      -- (1,0)
    };
    \draw [loosely dotted]          (1,0)             --   (3,0);
    \draw                           (3,0)             --   (4,0);
    \draw \foreach \x in {4,5}  {
      (\x+0,0)
      node [circle, draw, fill=black!50, inner sep=0pt, minimum width=4pt] {}
      --   (\x+1,0)};
    \draw [loosely dotted]          (6,0)             --   (8,0);
    \draw                           (8,0)             --   (9,0);
    \draw \foreach \x in {9,10}  {
      (\x+0,0)
      node [circle, draw, fill=black!50, inner sep=0pt, minimum width=4pt] {}
      --   (\x+1,0)};
    \draw [loosely dotted]          (11,0)            --   (13,0);
    \draw                           (13,0)            --   (14,0)
      node [circle, draw, fill=black!50, inner sep=0pt, minimum width=4pt] {}
      ;
    \draw [blue]                     (0,0.2)          --   (4,0.2);
    \draw [blue]                     (4,0.2)          arc  (115:63.7:11.82);
    \draw [blue]                     (14.22,-0.22)    --   (14.22,0.09);
    \draw [blue]                     (10,-0.2)        --   (14.22,-0.2);
    \draw [blue]                     (10,-0.2)        --   (9.8,0);
    \draw [blue]		     (9.8,0)          arc  (58:116.6:4.34);
    \draw [blue]                     (5.56,0.2)       --   (9,0.2);
    \draw [blue]		     (9,0.2)          arc  (90:0:0.2);
    \draw [blue]                     (9.2,0)          --   (9.2,-0.11);
    \draw [blue]		     (9.2,-0.1)       arc  (-55:-125:8.2);
    \draw [blue]                     (-0.2,0.2)       --   (-0.2,-0.11);
    \draw [blue]                     (-0.2,0.2)       --   (0,0.2);
    \node  at (0, 0.6) {$1$};
    \node  at (3.8, -0.6) {$j\in[1, k)$};
    \node  at (5.2, -0.6) {$j^+$};
    \node  at (9.8,  -0.6) {$i\in[k, n)$};
    \node  at (10, 0.6) {$i^+$};
    \node  at (14,-0.6) {$n$};
\end{tikzpicture}
        \caption{Double cross flip}
      \end{figure}

  \end{description}

  We shall show that at least one of these situations must hold.
  Suppose for a contradiction that this is not the case.
  Then for every $k \in (1, n)$
\begin{equation}  \label{eq:no_edges_1}
   E[(N(1) \cap (1, k])^-, (N(n) \cap [k, n))^+] = \emptyset
\end{equation}
since (b) does not hold; and
\begin{equation}  \label{eq:no_edges_2}
    E[\{1\} \cup (N(1) \cap [k, n))^+,  (N(n) \cap [1, k))^+] = \emptyset
\end{equation}
since (a) and (c) do not hold.

Let $1 \leq s \le t$ be such that $\widetilde{\alpha}(G) + 1 = s + t$
and $G$ has no $(s, t)$-bipartite-hole.
Since $\widetilde{\alpha}(G) \ge 2$,
we have $s \le \frac{\widetilde{\alpha}(G) + 1}2 < \widetilde{\alpha}(G)$,
and hence
\[
  |N(1) \cap (1, 2]| =1 \leq s \le \delta(G) - 1 < |N(1) \cap (1, n]| = d(1).
\]
Therefore we can choose $k \in (1, n)$ such that $|N(1) \cap (1, k]| = s$.
  Equation~(\ref{eq:no_edges_1}) implies that $|N(n) \cap [k, n)| < t$.
  Since $|N(n) \cap [1, k)| + |N(n) \cap [k, n)| \ge \delta(G)$,
  we have $|N(n) \cap [1, k)| > \delta(G) - t \ge \widetilde{\alpha}(G) - t = s - 1$,
  and so $|N(n) \cap [1, k)| \ge s$.
  Now from (\ref{eq:no_edges_2}) we deduce $|N(1) \cap [k, n)| < t - 1$,
  hence $|N(1) \cap [k, n)| \le t - 2$.
  Finally, since $1$ is not adjacent to $n$, we have
  \[
    \delta(G) \le |N(1) \cap (1, k]| + |N(1) \cap [k, n)| \le s +t - 2 \le \delta(G) - 1,
  \]
  and this contradiction completes the proof.
\end{proof}

Next we consider edge-disjoint Hamilton cycles.
We need a preliminary lemma.
For graphs $F$ and $G$ with the same vertex set $V$, we define
$F \cup G = (V, E(F) \cup E(G))$ and $F - G = (V, E(F) \setminus E(G))$.

\begin{lemma}
  \label{thm:disjoint_preliminary}
  Suppose $H_1, \ldots, H_r$ are $r \ge 1$
  Hamilton cycles in a graph $G$ and let $H = H_1 \cup \ldots \cup H_r$.
  Then $\widetilde{\alpha}(G - H) + 1 \le (r + 1)(\widetilde{\alpha}(G) + 1)$.
\end{lemma}

\begin{proof}
  Let $1 \leq s \le t$ be such that $\widetilde{\alpha}(G) + 1 = s + t$
  and $G$ has no $(s, t)$-bipartite-hole. Let $U, W \subseteq V(G)$ be disjoint sets
  of size $s$ and $2rs + t$ respectively. Now
\[
  |W \setminus \Gamma_H(U)|
  \ge |W| - \sum_{i=1}^r |\Gamma_{H_i}(U)| \geq 2rs + t - 2rs = t.
\]
  But $G$ has no $(s, t)$--bipartite-hole, so $G - H$ has no $(s, 2rs + t)$--bipartite-hole.
  Finally, we see that
  $\widetilde{\alpha}(G - H) + 1 \le s + 2rs + t \le (r + 1)(\widetilde{\alpha}(G) + 1)$,
  since $s \le \frac{\widetilde{\alpha}(G) + 1}2$.
\end{proof}

\disjointcycles*
\begin{proof}
  We sequentially find edge-disjoint Hamilton cycles $H_1, H_2, \ldots$.
  Let $0 \le i \le r$ and suppose we have found $H_1, \ldots, H_i$.
  Let $G_i = G - \cup_{j \le i}H_j$.
  Then by Lemma~\ref{thm:disjoint_preliminary}
  \begin{align*}
    \widetilde{\alpha}(G_i)
    \le (i+1)(\widetilde{\alpha}(G)+1) - 1
    \le (r+1)\widetilde{\alpha}(G) + r
    \le \delta(G) - 2r
    \le \delta(G_i).
  \end{align*}
  Hence by Theorem~\ref{thm:extremal} we can find $H_{i+1}$ edge-disjoint from
  $H_1, \ldots H_i$.
\end{proof}

A \emph{certificate} that $\widetilde{\alpha}(G) \ge k$
consists of pairs $(S_i, T_i)$
for $i=1,\ldots,\lfloor k/2\rfloor$
such that $S_i, T_i \subseteq V(G)$, $S_i \cap T_i = \emptyset$,
$E(S_i, T_i) = \emptyset$,
and $|S_i| = i$, $|T_i| = k - i$.

\algorithm*
\begin{proof}
  First check if $G$ is connected.
  If not, pick two connected components and note that each has size at least
  $\delta(G) + 1$.
  For each $i=1,\ldots, \lfloor (\delta(G) +1)/2\rfloor$,
  any $i$ vertices from one of these components together with
  any $\delta(G) + 2 - i$ from the other form a bipartite hole,
  and hence we can find a certificate that $\widetilde{\alpha}(G) \ge \delta(G) + 2$.
  So we can assume that $G$ is connected.

  Maintain a path $P$ with initial length at least two.
  The algorithm performs at most $n$ steps,
  and the length of $P$ increases with each one.
  On each step, check if a terminal vertex of $P$ has a neighbour outside $V(P)$,
  and if so extend $P$.
  Otherwise, following the proof of Theorem \ref{thm:extremal},
  we can either find a sequence of bipartite holes forming a certificate as required and
  halt, or close $P$ to form a cycle.
  This cycle is either Hamiltonian and then the algorithm halts,
  or from the connectivity of $G$ we can attach an edge $xy$
  with $x \in V(P)$ and $y \not\in V(P)$ to obtain a
  strictly longer path starting from $y$ and spanning $V(P) \cup \{y\}$.

  Each step takes $O(n^2)$ time, so the total time spent is $O(n^3)$.
\end{proof}

%%%%%%%%%%%%%%%%%%%%%%%%%%%%%%%%%%

\section{Application to dense random graphs}

The following result is phrased to cover the existence of one Hamilton cycle,
and of many.

\begin{lemma}
  \label{thm:random_elaborate}
  Fix $0<\epsilon<1$ and let $0 \leq p=p(n) \le 1 - \epsilon$ for all $n$.
  Given $r =r(n) \ge 1$, let $A_r$ be the event that $G(n, p)$ contains at least $r$
  edge-disjoint Hamilton cycles, and let $A^c_r$ be the complementary event.
  Then
  \[
  n\log (1-p) \le \log \mathbb{P}(A^c_r) \le n\log (1-p) +
  (2+o(1)) \, r\sqrt{n}\log n.
  \]
\end{lemma}
\begin{proof}
  Let $G \sim G(n, p)$, $t = \lceil \sqrt n \rceil$ and
  $d = r(2t) + 3r - 3$.
  From Theorem~\ref{thm:disjointcycles} we have
  $\{\widetilde{\alpha}(G) \le 2t\} \cap \{\delta(G) \ge d\} \subseteq A_r$,
  so
  \[
  \{\delta(G) = 0\} \subseteq A^c_r
  \subseteq \{\widetilde{\alpha}(G) > 2t\} \cup \{\delta(G) < d\}.
  \]
  Clearly $\mathbb{P}(\delta(G) = 0) \ge (1-p)^n = \exp( n\log(1-p))$.
Also, the probability that vertex $n$ has degree at most $d-1$
is at most the expected number of $(d-1)$-subsets of $[n-1]$ such that each other vertex is not adjacent to $n$.  Thus
\begin{align*}
    \mathbb{P}(\delta(G) < d)
    &\le n \binom{n-1}{d-1}  (1-p)^{(n-1)-(d-1)}
    \le n^d(1-p)^n \epsilon^{-d}\\
    &= \exp(n\log(1-p) + d(\log n+ \log(1/\epsilon))).
\end{align*}
Further
\begin{align*}
  \mathbb{P}(\widetilde{\alpha}(G) > 2t)
    &\le \mathbb{P}(G \mbox{ has a } (t,t)\mbox{-bipartite-hole})\\
    &\le {n \choose t}^2(1-p)^{t^2}
    \le \left( \frac{en}{t}\right)^{2t}(1-p)^{t^2}
    \le e^{2t}n^t(1-p)^n\\
    &= \exp(n\log (1-p) + \sqrt n (\log n + O(1)));
  \end{align*}
and the required upper bound on $\log \mathbb{P}(A^c_r)$ follows
since $d = (2+ o(1)) \, r\sqrt{n}$.
\end{proof}

\noindent
Theorem~\ref{thm:random_disjoint} and
Corollary~\ref{thm:random_cycle} follow directly from Lemma~\ref{thm:random_elaborate}.

%%%%%%%%%%%%%%%%%%%%%%%%%%%%%%%%%

\section{Complexity of computing and approximating $\widetilde\alpha(G)$}
\label{sec.hard}

Computing $\widetilde\alpha(G)$ is closely related to the following problem:

\smallskip
\noindent
\textbf{Maximum Balanced Complete Bipartite Subgraph (BCBS):}\\
\emph{Instance}: A positive integer $k$ and a bipartite graph $G$
with parts $A$ and $B$ where $|A|=|B|$; \\
\emph{Question}: Does $G$ contain a complete bipartite graph with $k$ vertices
in each part; that is, does $G$ have a subgraph $K_{k,k}$?
\smallskip

We use lemma 2.2 of \cite{alon1994}.
By that result the BCBS problem is NP-complete.
Also, BCBS is problem [GT24] in~\cite{michael_johnson}.

\smallskip
\noindent
\textbf{Bipartite Hole-Number (BHN):}\\
\emph{Instance}: A positive integer $k$ and a graph $G$; \\
\emph{Question}: Is $\widetilde{\alpha}(G) \geq k$?
\smallskip

To compare the two problems we introduce the following lemma:
\begin{lemma}
  \label{BCBS_to_BHN}
  Given a graph $G$ and an integer $k$,
  let $G'$ be formed from $G$ by adding a disjoint copy of $K_{k-1,2k}$;
  and let $G^\phi_k$ be the complement of $G'$.
  Then $K_{k, k} \subseteq G$ if and only if
  $\widetilde\alpha(G^\phi_k) \ge 2k$.
\end{lemma}
\begin{proof}
  We see that $G^\phi_k$ has an induced copy of $K_{k-1} \cup K_{2k}$,
  and so it has an $(s,2k-s)$-bipartite-hole for each $s=1,\ldots,k-1$.
  Thus $\widetilde{\alpha}(G^\phi_k) \geq 2k$ if and only if $G^\phi_k$
  has a $(k,k)$-bipartite-hole;
  and that happens if and only if
  the complement of $G$ has a $(k,k)$-bipartite-hole,
  if and only if $G$ has a subgraph $K_{k,k}$.
\end{proof}

The next proposition follows as a corollary.
\begin{prop}
  The BHN problem is NP-complete.
\end{prop}

In fact, a stronger statement could be given:
the BHN problem is hard to approximate.
To this end we use a result from \cite{hardness_of_approximation}
stating that the BCBS problem cannot be approximated within a factor
of $2^{(\log n)^{\delta}}$ for some $\delta > 0$,
unless 3-SAT can be solved in $O\left(2^{n^{3/4 + \epsilon}}\right)$ time
for every $\epsilon > 0$.
The widely believed Exponential Time Hypothesis (ETH)
states that 3-SAT cannot be solved in $2^{o(n)}$ time,
which provides strong evidence for the inapproximability of BCBS.
Lemma~\ref{BCBS_to_BHN} allows us to directly translate these results
to hardness of approximating $\widetilde{\alpha}(G)$:

\begin{prop}
  There exists $\delta > 0$ such that
  $\widetilde{\alpha}(G)$ cannot be approximated within a factor
  of $2^{(\log n)^{\delta}}$
  provided that 3-SAT $\notin DTIME\left(2^{n^{3/4 + \epsilon}}\right)$
  for some $\epsilon > 0$.
\end{prop}

%%%%%%%%%%%%%%%%%%%%%%%%%%%%%%%%%%

\section{Concluding remarks}
In this paper we presented a tight sufficient condition for Hamiltonicity,
Theorem~\ref{thm:extremal},
and used that result to prove an extension concerning the existence
of $r$ edge-disjoint Hamilton cycles, Theorem~\ref{thm:disjointcycles}.
As an application of these theorems,
we proved results on disjoint Hamilton cycles in dense random graphs.
It was pointed out to one of us by Michael Krivelevich that results
from~\cite{Krivelevich_hamilton}
should allow us to extend Corollary~\ref{thm:random_cycle}
to much lower edge-probabilities $p(n)$,
down to near the threshold for the necessary minimum degree;
and indeed this is the case as long as
$\frac{np(n)\log\log\log\log n}{\log n \log\log\log n} \to \infty$,
see the appendix in arXiv:math.CO/1604.00888.

We are not aware of any examples where the inequality in
Theorem~\ref{thm:disjointcycles} is sharp for $r \ge 1$.
It would be interesting to find such examples or relax the condition.

%%%%%%%%%%%%%%%%%%%%%%%%%%%%%%%%%%
\newpage
\section*{Appendix: Note on the probability of containing a Hamilton cycle
  around the threshold}
The \emph{external neighbourhood} of a set $S \subseteq V$
is denoted by $N(S)$, that is
\[
N(S) = \{v \in V \setminus S : v \text{ is adjacent to some } u \in S\}.
\]
We use the following result by Hefetz et al. \cite{Krivelevich_hamilton}.
\begin{thm}
  Suppose $12 \le d \le e^{\sqrt[3]{\log n}}$
  and $m = m(n, d) = \frac{\log n \log \log \log n}{\log d \log \log n}$.\
  If $G = (V, E)$ is a graph such that
  \begin{enumerate}
  \item [P1]
    for each $S \subseteq V$ if $|S| \le \frac{n}{dm}$,
    then $|N(S)| \ge d|S|$, and
  \item [P2]
    for every disjoint $A, B \subseteq V$ if $|A|, |B| \ge \frac{n}{4130m}$,
    then $E(A, B) \neq \emptyset$,
  \end{enumerate}
  then $G$ is Hamiltonian for large $n$.
\end{thm}

\begin{lemma}
  If $p \ge 16520 \frac{m \log (e4130 m)}{n}$,
  then $G(n, p)$ fails to satisfy P2 with probability
  less than $e^{-2np}$ for large $n$.
\end{lemma}
\begin{proof}
  It is sufficient to show that the claim holds for each
  $|A|$ and $|B|$ of size $k = \lceil \frac{n}{4130m} \rceil$.
  The number of choices for $A, B$ is at most
  \[ {n \choose k}^2
  \le \left(\frac{en}{k}\right)^{2k}
  =\exp \left(2k \log \frac{en}{k}\right).
  \]
  The probability that a fixed pair $(A, B)$ is edgeless is
  $(1-p)^{k^2} \le \exp(-pk^2)$.
  Now use the union bound to deduce that
  the log of the probability that P2 fails is at most
  \[
  2k \log \frac{en}{k} - pk^2 \le -\frac12pk^2 \le -2pn,
  \]
  where the first inequality holds since
  $p \ge \frac{4 \log(e4130m)}{\frac{n}{4130m}}$
  $\ge \frac{4 \log \frac{en}{k}}{k}$,
  and the second holds for large $n$.
\end{proof}

\begin{lemma}
  If $p \ge (8d + 12) \frac{\log n}{n}$,
  then $G(n, p)$ fails to satisfy P1 for some $|S| \ge 2$
  with probability at most $e^{-\frac{3}{2}pn}$
  for large $n$.
\end{lemma}
\begin{proof}
  The probability that a set $S$ of size $s$, $2 \le s \le \frac{n}{dm}$,
  is not expanding is at most
  \[
    p_0 = {n \choose s}{n - s \choose \lfloor ds \rfloor}(1-p)^{s(n-s-ds)}
    \le \exp\left((d+1)s \log n \right) \exp \left(-\frac 78 nsp \right),
  \]
  since $n - s - ds \ge \frac78n$ for large $n$.
  Note that $\frac{ax+b}{cx+d}$ is increasing in $x$ on $(-\frac dc, \infty)$
  iff $ad - cb > 0$ and decreasing on $(-\frac dc, \infty)$ otherwise,
  hence $\frac{(d+1)s + 1}{\frac{7}{8}s - \frac{3}{2}}$
  is decreasing in $s$ on $(\frac{12}7, \infty)$,
  and therefore $\frac{(d+1)s + 1}{\frac{7}{8}s - \frac{3}{2}} \le 8d + 12$,
  since $s \ge 2$.
  Finally,
  \begin{align*}
    &pn \ge (8d+12) \log n
    \ge \frac{(d+1)s + 1}{\frac{7}{8}s - \frac{3}{2}} \log n \\
    \Rightarrow \hspace{5pt} &\log n + (d+1)s \log n -\frac 78 nsp \le -\frac 32pn \\
    \Rightarrow \hspace{5pt} &n p_0 \le \exp\left(- \frac32pn \right).
  \end{align*}
  Taking the union bound over all $s \ge 2$ yields the statement.
\end{proof}

We see that the failure probability is dominated by the probability
of a single low-degree vertex.

\begin{lemma}
  Suppose $\frac{d \log n}{n} << p < \frac12$.
  Then
  \[
  \mathbb{P}(\delta(G(n, p)) \le d) \le \exp(-pn(1+o(1))).
  \]
\end{lemma}
\begin{proof}
    \begin{align*}
    \mathbb{P}(\delta(G) < d)
    &\le n {n-1 \choose d-1}(1-p)^{(n-1)-(d-1)} \le n^d e^{-pn} 2^d\\
    &=\exp(-pn + d \log(2n)) = \exp(-pn(1 + o(1))).
    & \hspace{66pt}\qedhere
  \end{align*}
\end{proof}

For $k \ge 1$ the notation $\log^{(k)}n$ stands for $\log n$ if $k=1$
and $\log (\log^{(k-1)}n)$ otherwise.
To wrap things up, note that $\log m = \Theta(\log^{(2)} n)$,
hence $m \log m = \Theta(\frac{\log n \log^{(3)} n}{\log d})$.
By solving the equation $d \log n = m \log m$
we get $d \log d = \Theta(\log^{(3)} n)$
and hence $d = \Theta(\frac{\log^{(3)} n}{\log^{(4)} n})$.
Therefore if $p = \frac{\log n\log^{(3)} n}{n\log^{(4)} n}\omega(1)$,
the probability of $G(n, p)$ failing to contain a Hamilton cycle is
$e^{-np(1+o(1))}$.

\bibliographystyle{alpha}
\bibliography{ham_random}

\end{document}